\documentclass[12pt,reqno]{amsart}

\usepackage{amsfonts}
\usepackage{amscd}
\usepackage{amssymb}
\usepackage{amsfonts}
\usepackage{amscd}
 \usepackage{amsmath} 
\usepackage{mathrsfs}
\usepackage{enumerate}
\usepackage{float} % To use [H]
\usepackage[pdftex]{graphicx}
\allowdisplaybreaks

\usepackage{hyperref}

\date{\today}

\newtheorem{thm}{Theorem}[section]
\newtheorem{cor}[thm]{\bf{Corollary}}

\newtheorem{prop}[thm]{Proposition}
\theoremstyle{definition}

\theoremstyle{remark}
\newtheorem{rem}[thm]{Remark}

\numberwithin{equation}{section}

\newcommand{\R}{\mathbb R}

\newcommand{\Na}{\mathbb N}

\newcommand{\C}{{\mathbb C}}

\newcommand{\om}{ \omega}

\newcommand{\K}{\kappa }

\renewcommand{\Re}{\operatorname{Re}}
\renewcommand{\Im}{\operatorname{Im}}
 
\usepackage{color}

\newcommand{\red}[1]{\textcolor{red}{#1}}

\title[Uncertainty principle]
{An uncertainty principle for spectral projections on \\ rank one symmetric spaces of noncompact type }

\author[ Ganguly and Thangavelu]{ Pritam Ganguly and Sundaram Thangavelu}

\address[P. Ganguly, S. Thangavelu]{Department of Mathematics\\
Indian Institute of Science\\
560 012 Bangalore, India}

\email{pitamg@iisc.ac.in, veluma@iisc.ac.in}

 \date{}
 
\begin{document}

\maketitle

\begin{abstract} 
	Let $G $ be a noncompact semisimple Lie group with finite centre. Let $X=G/K$ be the associated Riemannian symmetric space and assume that $X$ is of rank one. The generalized spectral projections associated to the Laplace-Beltrami operator are given by $P_{\lambda}f =f\ast \Phi_{\lambda}$, where $\Phi_{\lambda}$ are the elementary spherical functions on $X$. In this paper, we prove an Ingham type uncertainty principle for $P_{\lambda}f$. Moreover, similar results are obtained in the case of generalized spectral projections associated to Dunkl Laplacian.   
 \end{abstract}

\section{Introduction}
An old paper of Ingham \cite{I} written in 1934 investigates the admissible decay of the Fourier transform of a compactly supported function on $ \R.$ Since any decay of the form $ |\hat{f}(y)| \leq C e^{-a|y|}, a> 0 $ is ruled out, due to the holomorphic extendability of $ f $, Ingham  considered a slightly slower decay  $ |\hat{f}(y)| \leq C e^{-a|y| \, \theta(|y|)}$ where $ \theta:[0,\infty)\rightarrow [0,\infty) $ is a decreasing function vanishing at infinity. He proved that this kind of decay is admissible for a compactly supported function $ f $ if and only if $ \int_1^\infty \theta(t) t^{-1} dt < \infty.$  It follows that under the assumption 
$ \int_1^\infty \theta(t) t^{-1} dt = \infty,$ any nontrivial function $ f $ for which $ |\hat{f}(y)| \leq C e^{-a|y|\, \theta(|y|)}$ cannot vanish on the complement of any compact set. We can view this as a result on the decay of the generalized spectral projections associated to $ \Delta $ on $ \R.$ Indeed, by defining $ P_\lambda f(y) = \hat{f}(\lambda) e^{i\lambda y}+ \hat{f}(-\lambda) e^{-i\lambda y}, $ we see that $ P_\lambda f $ are eigenfunctions of $ \Delta $ with eigenvalues $ -\lambda^2 $ and the decay $ |P_\lambda f(y)| \leq C e^{-a|\lambda| \theta(|\lambda|)}$ where $ \int_1^\infty \theta(t) t^{-1} dt = \infty $ is ruled out if $ f $ vanishes on the complement of any compact set. We remark that in the above statement we can replace the complement of a compact set by any open set $ V.$\\

Ingham's theorem  has received considerable attention in recent years and  analogues have been proved for Fourier series \cite{BSR} and Fourier transforms on symmetric spaces \cite{BPR} and nilpotent Lie groups \cite{BGST, BSR}.  In this note we would like to recast some of the results as uncertainty principles for generalized spectral projections associated to Laplace-Beltrami operators on  Riemannian symmetric spaces. Recall that such a space is of the form $ X = G/K $  where $ G $ is a semisimple Lie group and $ K $ a maximal compact subgroup.  When the sectional curvature of the underlying  Riemannian manifold is positive, which happens when $ G $ is compact, we are in the setting of  compact symmetric spaces. In this case, we have already  proved a version of Ingham's theorem for the spectral projections for the Laplace-Beltrami operator  $ \Delta_X$ in \cite{GT}. In the same work, we have also treated the spectral projections associated to Hermite and special Hermite operators. \\

Here we would like to treat the remaining case of non compact Riemannian symmetric spaces. We first assume that the sectional curvature is negative. The generalized spectral projections  $ P_\lambda f $ associated to $ \Delta_X $  are given by convolutions with elementary spherical functions, thus $ P_\lambda f = f \ast \Phi_\lambda $  where $ \Phi_\lambda $ are the elementary bounded spherical functions on $ X.$ Under the assumption that symmetric space $ X $ is of rank one, we prove the following result.

\begin{thm}
	\label{thm1.1}
Let $ \theta $ be a positive decreasing function defined on $ [0,\infty) $ that vanishes  at infinity. Assume that $ \int_1^\infty \theta(t) \, t^{-1} dt = \infty.$ Let $ X = G/K $ be a rank one symmetric space of non compact type.  Let $ f \in L^1(X) $ be a nontrivial function vanishing on an open set $ V.$  Then 
the estimate  $\displaystyle \sup_{x \in V} |f \ast \Phi_\lambda(x)| \leq C e^{-\lambda  \theta(\lambda)} $ cannot hold uniformly for all $\lambda$.  Moreover, if $V$ contains  the identity, then the uniform estimate $\displaystyle \sup_{x \in V^c} |f \ast \Phi_\lambda(x)| \leq C e^{-\lambda  \theta(\lambda)} $,  is also not possible.   
\end{thm}

We show that in order to prove the first part of the above theorem, it is enough to prove a version of Ingham's theorem for Jacobi transform. As in the case compact symmetric spaces (see \cite{GT}) this is achieved by considering the spherical means on the symmetric spaces. For the second part of the theorem, we need a refined version of Ingham's theorem for the Helgason Fourier transform on $ X.$

 As an immediate consequence of the above theorem, the usual version of Ingham's theorem where the decay condition is assumed on the Fourier transform side,  can be obtained. In order to describe that we need some more notations. For $f\in L^1(X)$, the Helgason Fourier transform of $f$ defined on $\mathfrak{a}^{*}\times K/M$, is denoted by $\tilde{f}$. The generalized spectral projections can be represented in terms of $\tilde{f}$ as follows:
	$$	f\ast \Phi_\lambda(x)=\int_{K/M}e^{(-i\lambda+\rho)A(x,b)}\tilde{f}(\lambda,b)db.$$ For the unexplained notations and more details we refer the reader to section 3. Now from the above formula it is easy to see that $|\tilde{f}(\lambda,b)|\leq Ce^{-|\lambda|\theta(|\lambda|)},~\forall \lambda\in \mathfrak{a}^{*}$ implies that $|f\ast\Phi_{\lambda}(x)|\leq Ce^{-\lambda\theta(\lambda)}, ~\lambda>0$ whence we have the following result:
	\begin{cor}
		Let $ \theta $ be a positive decreasing function defined on $ [0,\infty) $ that vanishes  at infinity. Let $ X = G/K $ be a rank one symmetric space of non compact type.  Let $ f \in L^1(X) $ be a nontrivial function vanishing on an open set $ V$ whose Helgason Fourier transform satisfies $$|\tilde{f}(\lambda,b)|\leq Ce^{-|\lambda|\theta(|\lambda|)},~\forall \lambda\in \mathfrak{a}^{*}~\text{and} ~\forall b\in K/M.$$ If $ \int_1^\infty \theta(t) \, t^{-1} dt = \infty$, then $f=0$.  
	\end{cor}
	
 We remark that the above result has been discussed in \cite[Theorem 4.2]{BPR} for  noncompact Riemannian symmetric spaces of arbitrary  rank.   \\

We now turn our attention to the case of flat symmetric spaces. Let $ G $ be a non compact semisimple Lie group with finite centre with Lie algebra $ \frak{g}.$  Let $ K $ be a maximal compact subgroup of $ G $ and let $ \frak{g} = \frak{t} \oplus \frak{p} $ be the corresponding Cartan decomposition. The group $ K $ acts on $ \frak{p} $ via the adjoint representation. Forming the semidirect product $ G_0 = K \ltimes \frak{p},$ the flat symmetric space is defined as $ G_0/K $ which can be identified with $ \frak{p}.$ Here the Cartan motion group $G_0$ is of general rank. For more on flat symmetric spaces and its connection with $ G/K $ we refer to Orsted-Ben Said \cite{BO}.
Inside $ \frak{p} $ we choose a maximal abelian subspace $ \frak{a}$.
We denote by $ \Sigma $   the set of restricted roots with multiplicities $ m_\alpha, \alpha \in \Sigma $ and consider it  to be a subset of $ \frak{a}$  by identifying $ \frak{a} $  with $ \frak{a}^\ast $ via the Killing form  $ B.$  We let $ W $ stand for the  the Weyl group, acting on $ \frak{a}.$ It turns out that $ W $ is an example of a finite reflection group associated to a root system $ R $ and the function $ \kappa $ defined by $ \kappa_\alpha = \frac{1}{4} \sum_{ \beta \in \Sigma \cap \R \alpha} m_\beta $ is a multiplicity function. Thus we are in the setting of Dunkl operators introduced and studied by Dunkl \cite{D}. \\

Associated to the root system $ R $ and the multiplicity function $ \kappa $ we have the Dunkl transform $ \mathcal{F}_\kappa $ and the Dunkl Laplacian $ \Delta_\kappa.$  The spherical Fourier transform of $ K $-invariant functions on $ G_0/K $ are given by integrating against spherical functions which are expressible in terms of the  Dunkl exponential $ E_\kappa(ix, \xi).$ The action of the radial part of the Laplace-Beltrami operator $ \Delta$ acting on a $ K$-invariant function $ f $ on $ G_0/K $ is given by $ \Delta_\kappa f_{\frak{a}} $ where $f_{\frak{a}} $ is the restriction of $ f $ to $ \frak{a},$  see Remarks 4.27 (2) in de Jeu \cite{J1}. Thus the following result includes an Ingham's theorem for the generalized spectral projections associated to $ \Delta$ on $ G_0/K.$\\

Let $ \Delta_\kappa $ be the Dunkl Laplacian associated to a general root system $ R$ and a multiplicity function $ \kappa.$ Let $ P_\lambda f = f \ast_\kappa \varphi_{\K,\lambda} , \lambda > 0 $ be the generalized spectral projections associated to $ \Delta_\kappa.$ In the above, $ \varphi_\lambda $ are Bessel functions of certain order and $ \ast_\kappa $ stands for the Dunkl convolution. We refer to Section 4 for more details on all the unexplained concepts and terms related to Dunkl analysis. We prove the following:

\begin{thm}
	\label{thm1.3}
Let $ \theta $ be a positive decreasing function defined on $ [0,\infty) $ that vanishes  at infinity. Assume that $ \int_1^\infty \theta(t) \, t^{-1} dt = \infty.$   Let $ f \in L^1(\R^n, h_\kappa^2 dx) $ be a nontrivial function vanishing on an open set $ V.$  Then 
the estimate  $\displaystyle \sup_{x \in V} |f \ast_\kappa \varphi_{\K, \lambda}(x)| \leq C e^{-\lambda  \theta(\lambda)} $ cannot hold uniformly for all $\lambda$ for the generalized spectral projections associated to the Dunkl Laplacian.  Moreover, if $V$ contains  the origin, then the uniform estimate $\displaystyle \sup_{x \in V^c} |f \ast_\kappa \varphi_{\K,\lambda}(x)| \leq C e^{-\lambda  \theta(\lambda)} $  is also not possible.  
\end{thm}

As in the case of the symmetric space $ G/K $ the above theorem will be proved by making use of an Ingham's  theorem for the Hankel transform. This reduction is facilitated by considering the spherical means $ f \ast_\kappa \mu_r(x) $ on $ \R^n$ where $ \mu_r $ is the normalised surface measure on the sphere $ |x| = r.$ As in the case of standard spherical means, which corresponds to $ \kappa = 0$ these are Dunkl multipliers given by the Bessel functions $ \varphi_\lambda(r)$ which explains the connection with the Hankel transform, see Section 4.

Also similar to the symmetric space case, we have a version of Ingham's theorem for Dunkl transform which is an easy consequence of the above theorem.
\begin{cor} Let $ \theta $ be a positive decreasing function defined on $ [0,\infty) $ that vanishes  at infinity. Let $f\in L^1(\mathbb{R}^n, h^2_{\kappa}(x)dx)$ be such that its Dunkl transform $\mathcal{F}_{\kappa}f$ satisfies 
	$$|\mathcal{F}_{\kappa}f(\xi)|\leq Ce^{-|\xi|\theta(|\xi|)},~\forall \xi\in \mathbb{R}^n.$$
	If $f$ vanishes on a nonempty open set and 
	$ \int_1^\infty \theta(t) \, t^{-1} dt = \infty$, then $f=0.$
\end{cor} 
Using a different approach, recently this result has been proved in \cite{BPP}.  Moreover, we remark that putting $\kappa=0$ in the above result we can get a version of Ingham's theorem for the Fourier transform on $\mathbb{R}^n$ which was proved in \cite{BSR} using a several variable version of the classical Denjoy-Carleman theorem for the quasi-analytic functions. Consequently, our result in the paper provides a new and simple proof of Ingham type uncertainty principle for the Fourier transform on $\mathbb{R}^n$.  

We conclude this introduction by briefly describing the organisation of the paper. In Section 2, we prove a version of Chernoff's theorem for Bessel and Jacobi operators. Using this, in Section 3, we prove a refined version of Ingham's theorem for the Helgason Fourier transform on rank one Riemannian symmetric spaces of noncompact type.  Making use of spherical means, we  then prove an  Ingham type uncertainty principle for the generalized spectral projections associated to the Laplace-Beltrami operator. Finally, in Section 4, we prove similar Ingham type results for the generalized spectral projections associated to Dunkl Laplacian.

\section{Chernoff's theorem for Bessel  and Jacobi operators}
In 1975, P. R. Chernoff proved an $L^2$ version of classical Denjoy-Carleman theorem which deals with quasi analytic functions on $\R^n.$ In his paper \cite{C}, he used iterates of the Laplacian to prove a sufficient condition for a smooth function to be quasi analytic (See \cite[Theorem 6.1]{C}). Because of its usefulness in proving Ingham type uncertainty theorems, study of this in different settings has received considerable attention in recent years. See the works \cite{BPP}, \cite{BGST}, \cite{GT} in this regard. Also it is worth pointing out that the full power of Chernoff's theorem is not required to prove Ingham type results. As can be seen from the above works, only a weaker version is sufficient for this purpose.

Our aim in this section is to prove a weaker version Chernoff's theorem for a   differential operator  $L$ under certain assumptions on its eigenfunction expansion. Later in this section we will see that typical examples of $L$ includes Bessel and Jacobi operators.   Suppose $w$ and $\tilde{w}$ be  positive functions on $\R^+$   and $ \R$ respectively. We assume that $ \tilde{w} $ is an even function.   Let $T:L^2(\R^+,w(r)dr)\rightarrow L^2(\R, \tilde{w}(\lambda)d\lambda)$ be a transformation defined by $$Tf(\lambda):=\int_{0}^{\infty}f(r)\psi_{\lambda}(r)w(r)dr, $$ where   $\psi_{\lambda}$'s are bounded eigenfunctions (not necessarily in $ L^2$) of the operator $L$ with eigenvalues $ (\lambda^2 +\delta^2) $  for some $ \delta >0,$ i.e., 
	$ L\psi_{\lambda}= (\lambda^2+\delta^2) \psi_{\lambda}.$ We also assume that $\psi_{\lambda}$ are normalised so that $\psi_{\lambda}(0)=1$ and  $ \psi_\lambda(r) = \psi_{-\lambda}(r).$ As a matter of fact, $Tf(\lambda)$ is an even function of $\lambda$.   We further assume that there are following versions of inversion and Plancherel formulas for this transform. Let  the inversion formula which is assumed to be held for a suitable dense class of functions,  read as  
	$$f(r)= c \int_{-\infty}^{\infty}Tf(\lambda)\psi_{\lambda}(r) \tilde{w}(\lambda)d\lambda$$ for some constant $ c >0.$ The  Plancherel formula for $f\in L^2(\R^+,w(r)dr)$ is assumed to be $$\|f\|^2_{L^2(\R^+,w(r)dr) }= c \int_{-\infty}^{\infty}|Tf(\lambda)|^2 \tilde{w}(\lambda)d\lambda .$$  Under all the assumptions  described above,  on the eigenfunction expansion of $L$, we prove the following version of Chernoff's theorem for $L$.
   
\begin{thm}
	\label{chernoff} Let $ f \in L^2(\R^+, w(r)dr) $ be such that $ L^mf \in L^2(\R^+, w(r)dr) $ for all $ m \in \Na $  and satisfies the Carleman condition 
	$ \sum_{m=1}^\infty  \| L^m f \|_2^{-1/(2m)} = \infty.$  Assume that $\tilde{w}$ has at most polynomial growth. Then $ f $ cannot vanish in a neighbourhood of $0 $ unless it is identically zero.
\end{thm} 
We remark that this theorem is actually a continuous version of Theorem 2.2 in \cite{GT} .  In order to prove this theorem, we need the following result due to de Jeu \cite{J}.
 \begin{thm}\label{carleman} Let $ \mu $ be a finite positive Borel measure on $ \R $ for which all the moments $ M(m) =\int_{-\infty}^\infty t^m d\mu $ are finite. If we further assume that the moments satisfy the Carleman condition $ \sum_{m=1}^\infty  M(2m)^{-1/2m} = \infty,$ then polynomials are dense in $ L^p(\R,d\mu), 1 \leq p < \infty.$
\end{thm}

 As remarked in   \cite{BSR} (see Remark 3.6), if the measure $\mu$ is even, then even polynomials are dense in $ L^p_e(\R, d\mu)$, the subspace of even functions in $ L^p(\R,d\mu).$ We will make use of this observation in the following proof. We remark that the proof given below is already present in \cite{BSR} but for the sake of convenience of the reader we reproduce it here.

\textbf{\textit{Proof of Theorem \ref{chernoff}:}}  Let $f$ be as in the statement of the theorem.   We consider the following measure $\mu_f$ defined on the Borel subsets of $\R$ by $$ \mu_f(E) =\int_E   |Tf(\lambda)| \tilde{w}(\lambda)d\lambda .$$ Hence it follows that 
 $$ \int_{-\infty}^\infty  t^{2m} d\mu_f(t) = \int_{-\infty}^{\infty}  \lambda^{2m}|Tf(\lambda)|\tilde{w}(\lambda)d\lambda \leq \int_{-\infty}^{\infty}  (\lambda^2+\delta^2)^{m}|Tf(\lambda)|\tilde{w}(\lambda)d\lambda. $$  But by Cauchy-Schwarz inequality,  we have 
 $$\int_{-\infty}^\infty  t^{2m} d\mu_f(t) \leq C_j \left(\int_{0}^{\infty}(\lambda^2+\delta^2)^{2(m+j)}|Tf(\lambda)|^2 \tilde{w}(\lambda)d\lambda\right)^{\frac{1}{2}} =C_j \|L^{(m+j)}f\|_2$$
 where $C_j^2= \int_{0}^{\infty}(\lambda^2+\delta^2)^{-2j}\tilde{w}(\lambda)d\lambda.$ Note that $C_j$ is finite for large enough $j $ in view of  our assumption that $\tilde{w}(\lambda)$ has polynomial growth in $\lambda.$ Now if we denote the $m^{th}$ order moment of the measure $\mu_f$ by $M(m)$, then from the above observations it follows that  
$$ \sum_{m=1}^\infty  M(2m)^{-1/2m} \geq  \sum_{m=1}^\infty C_j^{-1/2m} \| L^{(m+j)}f \|_2^{-1/2m} =\sum_{m=1}^\infty C_j^{-1/2m}\left(\| L^{2(m+j)}f \|_2^{-1/2(m+j)}\right)^{\frac{m+j}{m}}. $$  
 Hence from the hypothesis  $\sum_{m=1}\|L^{m}f\|_2^{-\frac{1}{2m}}=\infty,$  combined with Lemma 3.3 in \cite{BPR} it follows that $\sum_{m=1}^\infty  M(2m)^{-1/2m}=\infty.$ So, by Theorem \ref{carleman}, and the remark following it, even polynomials are dense in $L^1_e(\R,d\mu_f)$.   Since $f\in  L^2(\R,w(r)dr)$, by Plancherel $Tf\in L^2_e(\R,\tilde{w}(\lambda)d\lambda).$ Clearly the function $\varphi(\lambda)=\overline{Tf(\lambda)}$ is even and      $$ \int_{-\infty}^\infty | \varphi(\lambda)| d\mu_f(\lambda) = \int_0^\infty  |Tf(\lambda)|^2 \tilde{w}(\lambda) d\lambda <\infty, $$ proving that $\varphi\in L^1_e(\R,d\mu_f)$.  Now given $\varepsilon>0$, there exists an even  polynomial   $q$ such that 
 \begin{equation}
 \label{c1}
 \int_{0}^{\infty} |  \overline{Tf  (\lambda) } -q(\lambda)||Tf(\lambda)|\tilde{w}(\lambda)d\lambda<\varepsilon.
 \end{equation}  
But notice  that $|Tf(\lambda)|^2=(\overline{Tf(\lambda)}-q(\lambda))Tf(\lambda)+q(\lambda)Tf(\lambda).$ So  from the Plancherel formula, we see that 
$$\|f\|_2^2\leq\int_{0}^{\infty} |\overline{Tf(\lambda)}-q(\lambda)||Tf(\lambda)|\tilde{w}(\lambda)d\lambda+\int_{0}^{\infty}q(\lambda)Tf(\lambda)\tilde{w}(\lambda)d\lambda. $$ 
Now under the assumption that $f$ vanishes on $(0,\eta)$, for some $\eta>0$, it follows that, for $0<r<\eta$ and for any $m\in\mathbb{N}$ we have 
\begin{equation}
\label{c2}
L^mf(r)=c\int_{0}^{\infty}  (\lambda^2+\delta^2)^mTf(\lambda)\psi_{\lambda}(r)\tilde{w}(\lambda)d\lambda=0.
\end{equation}
  Since $ q $ is an even polynomial, we can find another polynomial $ p $ such that $ p(\lambda^2+\delta^2) = q(\lambda).$ Therefore, \ref{c2}  along with the fact that  $\psi_{\lambda}(0)=1$   gives    
$$  \int_{ 0}^\infty q(\lambda) \, Tf(\lambda)\tilde{w}(\lambda)d\lambda = \lim_{r \rightarrow 0} \int_{ 0}^\infty p ( \lambda^2+\delta^2 ) \, Tf(\lambda) \psi_{\lambda}(r)\tilde{w}(\lambda) d\lambda= 0 $$ which together with \ref{c1} proves that
$  \|f\|_2^2  < \varepsilon .$
As this is true for every  $ \varepsilon>0 , $ it follows that $ f=0 $ completing the proof of the theorem.   \qed

 \begin{rem} A close examination of the above proof shows that Theorem 1.2 is valid under the weaker assumption that $ \displaystyle\lim_{r \rightarrow 0} L^mf(r) = 0 $ for all $ m.$
 \end{rem}
 
In the following subsections, we discuss two very important examples of the operator $L$.  We will be using them later in this article.
\subsection{Bessel differential operator-Hankel transform}
Let $\alpha>-\frac{1}{2}$. Suppose $\Delta_{\alpha}$ stands for the operator 
$$\Delta_{\alpha}f(r)=r^{-(2\alpha+1)}\frac{d}{dr}\left(r^{2\alpha+1}\frac{d}{dr}f(r)\right).$$   Let  $ J_\alpha(t)  $ stand for  the Bessel function of order $ \alpha $ defined by
$$J_{\alpha}(t)=\sum_{k=0}^{\infty}\frac{(-1)^k2^{-\alpha-2k}t^{\alpha+2k}}{\Gamma(k+1)\Gamma(k+\alpha+1) }.$$  It is well known  that for a suitable choice of $ c_\alpha $ the  functions  $\psi^{\alpha}_{\lambda}(r) =  c_\alpha J_{\alpha}(\lambda r)(\lambda r)^{-\alpha} $ are   solutions of the equation    
\begin{align*}
(\Delta_{\alpha}+\lambda^2)f(r)=0,~ f(0)=1.
\end{align*}   
It is convenient to subtract a positive constant from $\Delta_{\alpha}$ and work with that. We define $ \Delta_{\alpha,a} :=\Delta_{\alpha}-a^2$ where $a \neq 0.$ Clearly $\psi^{\alpha}_{\lambda}$ are eigenfunctions of $\Delta_{\alpha,a} $ with eigenvalue $-(\lambda^2+a^2)$  and they are even as  a function of $ \lambda.$  Also one can easily check that $\Delta_{\alpha,a} $ is self-adjoint on $L^2(\R^+, r^{2\alpha+1}\, dr)$.
The Hankel transform of order $\alpha$ for a suitable function $f$ is defined as 
$$\mathcal{H}^{(\alpha)}f(\lambda)=\int_{ 0}^{\infty}f(r)\psi_{\lambda}^{\alpha}(r)r^{2\alpha+1}dr. $$ It is easy to see from the definition of Bessel function that $\psi_{-\lambda}^{\alpha}(r)=\psi_{\lambda}^{\alpha}(r)$ and hence $\mathcal{H}^{(\alpha)}f(\lambda)$ is an even function of $\lambda$. We remark that Dunkl transform of a radial function on $\R^n$ is given by Hankel transform of order $\alpha$ for some $\alpha$. For more details see Section 4.  
Now using the self-adjointness of the operator under consideration, we have 
$$\mathcal{H}^{(\alpha)}(\Delta_{\alpha,a}f)(\lambda)=-(\lambda^2+a^2)\mathcal{H}^{(\alpha)}f(\lambda).$$ 
The inversion and Plancherel formula for the Hankel transform are described in the following theorem. 
\begin{thm} Assume that  $\alpha>-\frac{1}{2}$. We have the following inversion and Plancherel theorem for the Hankel transform.
	\begin{enumerate}
		\item (Inversion) If $f\in \mathcal{S}(\R^+)$, then 
		$$f(r)=\int_{ 0}^{\infty}\mathcal{H}^{(\alpha)}f(\lambda)\psi_{\lambda}^{\alpha}(r)\lambda^{2\alpha+1}d\lambda.$$
		\item (Plancherel) The map $f\rightarrow \mathcal{H}^{(\alpha)}f$ extends as an isometry  of $L^2(\R^+,r^{2\alpha+1}dr)$ onto $L^2_e(\R, |\lambda|^{2\alpha+1}d\lambda)$. 
	\end{enumerate}
\end{thm}
With these, finally we are ready to give a version of Chernoff's theorem for   $ L = -\Delta_{\alpha,a}  $.   In fact, the proof is already done in Theorem \ref{chernoff}. Note that here $w(r)=r^{2\alpha+1}$ and $ \tilde{w}(\lambda)=|\lambda|^{2\alpha+1}$ both have polynomial growth and $\Delta_{\alpha,a} $ satisfies all the hypothesis of Theorem \ref{chernoff}.So, we have 
\begin{thm}
	\label{chernoffH} Let $ f \in L^2(\R^+, r^{2\alpha+1}dr) $ be such that $\Delta_{\alpha,a} ^mf \in L^2(\R^+,r^{2\alpha+1}  dr) $ for all $ m \in \Na $  and satisfies the Carleman condition 
	$ \sum_{m=1}^\infty  \| \Delta_{\alpha,a} ^m f \|_2^{-1/(2m)} = \infty.$ Then $ f $ cannot vanish in a neighbourhood of $0 $ unless it is identically zero.
\end{thm}   
 
\subsection{Jacobi operators  and Jacobi transforms} We briefly discuss some results from Jacobi analysis here. For more details we refer the reader to Koornwinder \cite{K}.

Let $\alpha,\beta,\lambda\in \mathbb{C}$ and $-\alpha\notin \mathbb{N}.$ The Jacobi functions $\varphi_{\lambda}^{(\alpha, \beta)} $ of type $(\alpha,\beta)$ are  solutions of 
the initial value problem
\begin{align*}
 	(\mathcal{L}_{\alpha,\beta}+ \lambda^2+ \varrho^2)\varphi_{\lambda}^{( \alpha ,\beta)}(x) =0,\,\,\, \varphi_{\lambda}^{( \alpha , \beta )}(0)=1 
\end{align*} where $\mathcal{L}_{\alpha,\beta}$ is  the Jacobi operator defined by 
$$\mathcal{L}_{\alpha,\beta}:=\frac{d^2}{dr^2}+((2\alpha+1)\coth r+(2\beta+1)\tanh r) 
\frac{d}{dr}$$ and $\varrho=\alpha+\beta+1.$ Thus Jacobi functions $\varphi_{\lambda}^{(\alpha,\beta)}$ are eigenfunctions of $\mathcal{L}_{\alpha,\beta}$ with eigenvalues $-(\lambda^2+\varrho^2).$ These are even functions on $ \R $ and are expressible in terms of hypergeometric functions. For certain values of the parameters $ (\alpha, \beta) $ these functions arise naturally as spherical functions on Riemannian symmetric spaces of noncompact type. We shall see this later.   

The Jacobi transform of a suitable function $f$ on $\R^+$ is defined as 
$$\tilde{f}(\lambda)=\int_{ 0}^{\infty}f(r)\varphi_{\lambda}^{(\alpha,\beta)}(r)\tilde{w}_{\alpha,\beta}(r)dr$$ where the weight function   $\tilde{w}_{\alpha,\beta}(r)=(2\sinh r)^{2\alpha+1}(2\cosh r)^{2\beta+1}.$ This is also called the Fourier-Jacobi transform of type $(\alpha,\beta).$ Since $\varphi_{- \lambda}^{(\alpha,\beta)}(r)=\varphi_{ \lambda}^{(\alpha,\beta)}(r)$, $\tilde{f}(\lambda)$ is an even function of $\lambda$. It can be checked that the operator $\mathcal{L}_{\alpha,\beta}$ is self-adjoint on $L^2(\R^+,\tilde{w}_{\alpha,\beta}(r)dr)$ and that
$$\widetilde{\mathcal{L}_{\alpha,\beta}f}(\lambda)=-(\lambda^2+\varrho^2)\tilde{f}(\lambda).$$
Under certain assumptions on $\alpha$ and $\beta$ the inversion and Plancherel formula for this transform take nice forms as described below.
\begin{thm}
	\label{hpi}
	Let $\alpha,\beta\in\R$, $\alpha>-1$ and $|\beta|\leq\alpha+1.$ Suppose $ c_{\alpha,\beta}(\lambda) $ denotes the Harish-Chandra $c $-function defined by 
	$$ c_{\alpha,\beta}(\lambda)=\frac{2^{\varrho-i\lambda}\Gamma(\alpha+1)\Gamma(i\lambda)}{\Gamma\left(\frac{1}{2}(i\lambda+\varrho)\right)\Gamma\left(\frac{1}{2}(i\lambda+\alpha-\beta+1)\right)}$$ 
	\begin{enumerate}
		\item (Inversion) For $f\in C_0^{\infty}(\R)$ which is even we have 
		$$f(r)=\frac{1}{\pi}\int_{ -\infty}^{\infty}\tilde{f}(\lambda)\varphi_{\lambda}^{(\alpha,\beta)}(r)|c_{\alpha,\beta}(\lambda)|^{-2}d\lambda$$ 
		\item (Plancherel) For $f,g\in C^{\infty}_0(\R)$ which are even, the following holds
		$$\int_{ 0}^{\infty}f(r)\overline{ g(r)}\tilde{w}_{\alpha,\beta}(r)dr=\frac{1}{\pi}\int_{ -\infty}^{\infty}\tilde{f}(\lambda)\overline{\tilde{g}(\lambda)}|c_{\alpha,\beta}(\lambda)|^{-2}d\lambda.$$  
	\end{enumerate}
The mapping $f\rightarrowtail \tilde{f}$ extends  as an isometry from  $L^2(\R^+,\tilde{w}_{\alpha,\beta}(r)dr)$ onto
 $L^2_e(\R^+,|c_{\alpha,\beta}(\lambda)|^{-2}d\lambda).$
\end{thm}
 
Notice that $L=-\mathcal{L}_{\alpha,\beta}$ satisfies the hypothesis of Theorem \ref{chernoff} with $w(r)=\tilde{w}_{\alpha,\beta}(r)$. Also the weight $ \tilde{w}(\lambda)=|c_{\alpha,\beta}(\lambda)|^{-2}$ is even and has polynomial growth (e.g., see \cite{K2}.).  Hence we have the following version of Chernoff's theorem for the Jacobi operator:
 \begin{thm}
 	\label{chernoffJ} 	Let $\alpha,\beta\in\R$, $\alpha>-1$ and $|\beta|\leq\alpha+1.$ Suppose $ f \in L^2(\R^+, \tilde{w}_{\alpha,\beta}(r)dr) $ is such that $ \mathcal{L}_{\alpha,\beta}^mf \in L^2(\R^+, \tilde{w}_{\alpha,\beta}(r)dr) $ for all $ m \in \Na $  and satisfies the Carleman condition 
 	$ \sum_{m=1}^\infty  \|  \mathcal{L}_{\alpha,\beta}^m f \|_2^{-1/(2m)} = \infty.$ Then $ f $ cannot vanish in a neighbourhood of $0 $ unless it is identically zero.
 \end{thm} 

In the following sections we use the above results to prove Ingham type uncertainty principles for spectral projections associated to Laplace-Beltrami operator and Dunkl-Laplacian.

\section{Ingham's theorem for  rank one symmetric spaces  }

\subsection{Preliminaries on Riemannian symmetric spaces of non-compact type}
In this section we briefly describe the harmonic analysis on rank one Riemannian symmetric spaces of noncompact type. General references for this section are the books of Helgason \cite{H1} and \cite{H2}. 

Let $G$ be a connected, noncompact semisimple Lie group with finite center and $K$ be the maximal compact subgroup of $G$. Suppose $X=G/K$ is the associated Riemannian symmetric space. Assume that $X$ is of rank one.  In view of  the Iwasawa decomposition   $G=KAN$ with $N$ nilpotent and $A$ one-dimensional, every $g\in G$ can be expressed uniquely as $g=k(g)\ exp A(g) n(g)$ where  $A(g)$ belongs to the Lie algebra of $A$. Let $\mathfrak{ g}$ and $\mathfrak{k}$ denote the Lie algebra of $G$ and $K$ respectively. Then the corresponding Cartan decomposition reads as $\mathfrak{ g}=\mathfrak{k}\oplus \mathfrak{p} .$ Let $\mathfrak{a}$ be the maximal abelian subspace of $\mathfrak{p}$. Since $G$ is of rank one, the dimension of $\mathfrak{a}$ is one. It is well known that the non-zero roots of the pair $(\mathfrak{ g},\mathfrak{a})$ are given by either $\{\pm\gamma\} $ or $\{\pm \gamma, \pm 2\gamma\}$ where $\gamma$ is a positive root with respect to a positive weyl chamber.  Let $\rho:=(m_{\gamma}+m_{2\gamma})/2$ where $m_{\gamma}$ and $m_{2\gamma}$ denote the multiplicities of the roots $\gamma$ and $2\gamma$ respectively. The Haar measure $dg$ on $G$ is given by 
$$\int_{G}f(g)dg=\int_K\int_{A}\int_N f(k a_t n)e^{2\rho t}dkdtdn.                                                                                  $$ The measure $dx$ on $X$ is induced from the Haar measure $dg$ via the relation 
$$\int_Gf(gK)dg=\int_X f(x)dx.$$

Let $o$ denote the identity $eK$ in $X=G/K$ where $e$ is the identity element of the group $G$. It is known that the tangent space of $X$ at the point $o$ can be identified with $\mathfrak{p}$. The restriction of the Killing form $\mathfrak{B}$ of $\mathfrak{g}$ on $\mathfrak{p}$ induces a $G$-invariant Riemannian metric on $X$ which is denoted by $d_{X}$. Using this metric we define the open ball of radius $l$ and centered at $gK$ by 
$$B(gK,l)=\{hK:h\in G, d_{X}(gK,hK)<l\}.$$

Suppose $M$ denotes the centralizer of $A$ in $K$. Then the function $A:X\times K/M\rightarrow \mathfrak{a}$ defined by $A(gK,kM)=A(k^{-1}g)$ is right $K$-invariant in $g$ and right $M$-invariant in $K$.  In what follows we denote the elements of $X$ and $K/M$ by $x$ and $b$ respectively. Let $\mathfrak{a}^*$ denote the dual of $\mathfrak{a}$ and $\mathfrak{a}^*_{\C}$ be its complexification. Here in our case $\mathfrak{a}^*$ and $\mathfrak{a}^*_{\C}$ can be identified with $\R$ and $\C$ respectively. For each $\lambda\in \mathfrak{a}^*_{\C}$ and $b\in K/M$, the function $x\rightarrow e ^{(i\lambda+\rho)A(x,b)}$ is a joint eigenfunction of all invariant differential operators on $X.$ For $f\in C^{\infty}_c(X)$, its Helgason Fourier transform is a function $\widetilde{f}$ on $\mathfrak{a}^*_{\C}\times K/M$ defined by 
$$\tilde{f}(\lambda, b)= \int_X f(x)e ^{(-i\lambda+\rho)A(x,b)}dx,~ \lambda\in \mathfrak{a}^*_{\C},~ b\in K/M . $$ Moreover, we know that if $f\in L^1(X)$ then $\widetilde{f}(.,b)$ is a continuous function on $\mathfrak{a}^*$ which extends holomorphically to a domain containing $\mathfrak{a}^*.$ The inversion formula for $f\in C^{\infty}_c(X)$ says that 
$$f(x)=c_{X}\int_{-\infty}^{\infty}\int_{K/M}\widetilde{f}(\lambda,b)e ^{(i\lambda+\rho)A(x,b)}|c(\lambda)|^{-2}dbd\lambda$$ where $d\lambda$ stands for usual Lebesgue measure on $\R$ (i.e., $\mathfrak{a}^*$) , $db$ is the normalised measure on $K/M$ and $c(\lambda)$ is the Harish Chandra $c$-function. The constant $c_X$ appearing in the above formula is explicit and depends on the symmetric space $X$ (See e.g., \cite{H2}). Also for $f\in L^1(X)$ with $\widetilde{f}\in L^1(\mathfrak{a}^*\times K/M,|c(\lambda)|^{-2}dbd\lambda)$,the above inversion formula holds for a.e. $x\in X.$ Furthermore, the mapping $f\rightarrow \widetilde{f}$ extends as an isometry of $L^2(X)$ onto $L^2(\mathfrak{a}^*_{+}\times K/M, |c(\lambda)|^{-2}d\lambda db) $ which is known as the Plancherel theorem for the Helgason Fourier transform.

\subsection{Ingham's theorem for the Helgason Fourier transform}  In this subsection we prove a version of Ingham's theorem for the Helgason Fourier transform.  In order to state the result we need to consider  certain irreducible representations of $K$ with $M$-fixed vectors. Suppose $\widehat{K}_0$ denotes the set of all irreducible unitary representations of $K$ with $M$ fixed vectors. Let $\delta\in \widehat{K}_0$ and $V_{\delta}$ be the finite dimensional vectors space on which $\delta $ is realised. We know that $V_{\delta}$ contains a unique normalised $M$-fixed vector (See Kostant\cite{Ks}). Let $v_1\in V_{\delta}$ be the $M$-fixed vector in $V_{\delta}$. Consider  an orthonormal basis $\{v_1,v_2,...,v_{d_{\delta}}\}$ for $V_{\delta}$ with $ v_1 $ being the unique $ M$-fixed vector. For $\delta\in \widehat{K}_0$ and $1\le j\le d_{\delta}$, we define  $$Y_{\delta,j}(kM)=(v_j, \delta(k)v_1),~ kM\in K/M.$$ It is clear that $Y_{\delta,1}(eK)=1$ and $Y_{\delta,1}$ is $M$- invariant. 
\begin{prop}[\cite{H2}]
	The set $\{Y_{\delta,j}:1\le j\le d_{\delta},\delta\in \widehat{K}_0\}$ forms an orthonormal basis for $L^2(K/M)$.  
\end{prop} 
We can get an explicit realisation of $\widehat{K}_0$ by identifying $K/M$ with the unit sphere in $\mathfrak{p}$. Denoting $\mathcal{H}^m$ to be the space of homogeneous harmonic polynomials of degree $m$ restricted to the unit sphere, we have the following spherical harmonic decomposition
$$L^2(K/M)=\displaystyle\oplus_{m=0}^{\infty}\mathcal{H}^m .$$ It is known that each $ V_\delta $ is contained in some $ \mathcal{H}^m $ and hence the functions $Y_{\delta,j}$ can be identified with spherical harmonics.

Given $\delta\in \widehat{K}_0$ and $\lambda\in \mathfrak{a}^*_{\C} (i.e., \C~ \text{in our case})$ we consider the spherical functions of type $\delta$ defined by 
$$\Phi_{\lambda,\delta}(x):=\int_K e^{(i\lambda+\rho)A(x,kM)}Y_{\delta,1}(kM)dk.$$ These are eigenfunctions of the Laplace-Beltrami operator $\Delta_X$ with eigenvalue $-{(\lambda^2+\rho^2)}.$ When $\delta$ is the unit representation, $Y_{\delta,1}=1.$ In this case $\Phi_{\lambda,\delta}$ is called spherical function, denoted by $\Phi_\lambda$. More precisely, 
$$\Phi_\lambda(x)=\int_K e^{(i\lambda+\rho)A(x,kM)} dk.$$ Note that these functions are $K$-biinvariant. The spherical functions can be expressed in terms of Jacobi functions. In fact, if $x=gK$ and $g=ka_rk^{'}$ (polar decomposition), $\Phi_{\lambda,\delta}(x)=\Phi_{\lambda,\delta}(a_r)$. Suppose $$\alpha=\frac12(m_{\gamma}+m_{2\gamma}-1),~ \beta=\frac12(m_{2\gamma }-1).$$ For each $\delta\in \widehat{K}_0$ there exists a pair of integers $(p,q)$ such that 
\begin{equation}
\Phi_{\lambda,\delta}(x)=Q_{\delta}(i\lambda+\rho)(\alpha+1)_p^{-1}(\sinh r)^p(\cosh r)^q \varphi_{\lambda}^{(\alpha+p,\beta+q)}(r)
\end{equation} where $\varphi_{\lambda}^{(\alpha+p,\beta+q)}$ are the Jacobi functions of type $(\alpha+p,\beta+q)$ and $Q_{
	\delta}$ are the Kostant polynomials given by 
$$Q_{\delta}(i\lambda+\rho)=\left(\frac{1}{2}(\alpha+\beta+1+i\lambda)\right)_{(p+q)/2}\left(\frac{1}{2}(\alpha-\beta+1+i\lambda)\right)_{(p-q)/2}.$$ In the above we have used the notation $(z)_m=z(z+1)(z+2)...(z+m-1).$ We also require the following result proved in Helgason \cite{H2}
\begin{prop}
	\label{fsp}
	Let $\delta\in \widehat{K}_0$ and $1\le j\le d_{\delta}$. Then we have 
	\begin{equation}
	\int_K e^{(i\lambda+\rho)A(x,k^{'}M)}Y_{\delta,j}(k^{'}M)dk^{'}=Y_{\delta,j}(kM)\Phi_{\lambda,\delta}(a_r),~x=ka_r\in X. 
	\end{equation}
\end{prop}
  We are now ready to state and prove our version of Ingham's theorem. In order to do so, given a suitable function $f$ on $X$ we consider the function
	$$\tilde{F}_{\delta,j}(\lambda):= Q_{\delta}(i\lambda+\rho)^{-1} \int_{K}\widetilde{f}(\lambda,kM)Y_{\delta,j}(kM)dk  $$ where $\delta\in\widehat{K}_0$ and $Q_{\delta}$ are as above.    The following result is the analogue of Theorem 5.1 in \cite{T} proved in the context of Hardy's theorem.  
 
\begin{thm}
	\label{sspthm}
	Let $f\in L^1(X)$ be such that $f$ vanishes on an open neighbourhood of the identity  $V$. Suppose  for each $\delta\in \widehat{K}_0$ and $1\leq j\leq d_{\delta}$ the following estimate holds    
	\begin{equation}
	\label{ssp}
	 \left| \tilde{F}_{\delta,j}(\lambda)\right|\leq C_{\delta,j}e^{- \lambda \theta( \lambda )},~\lambda>0
	\end{equation}   
   where $\theta$ is a positive  decreasing function on $[0,\infty)$ which vanishes at infinity. Then, if $\int_{1}^{\infty}\theta(t)t^{-1}dt=\infty$,  $f$ is identically zero.  
\end{thm} 
\begin{proof}
	Without loss of generality, we may assume that $f$ vanishes on an open ball $B(o,l)$.
	For $\lambda\in\R$, we denote 
	$$F_{\delta,j}(\lambda)=\int_{K}\widetilde{f}(\lambda,kM)Y_{\delta,j}(kM)dk. $$ Using the definition of Helgason Fourier transform, we obtain
	\begin{equation}
	\label{p1}
F_{\delta,j}(\lambda)=\int_{K}\int_{G/K}f(x)e^{(-i\lambda+\rho)A(x,kM)}Y_{\delta,j}(kM)dxdk.
	\end{equation}
   So an application of Fubini along with the formula stated in the Proposition \ref{fsp} reduces \ref{p1}     to 
	\begin{equation}
	\label{p2}
	F_{\delta,j}(\lambda)=\int_{G/K}f(x)Y_{\delta,j}(kM)\Phi_{\lambda,\delta}(a_r)dx. 
	\end{equation}  
	Now we define $$f_{\delta,j}(x)=\int_{K}f(k'x)Y_{\delta,j}(k'M)dk',\ x\in X.$$ But since $f$ is right-$K$-invariant, it follows that $f_{\delta,j}$ is a $K$-biinvariant function on $G$. Moreover, note that given $x=gK\in B(o,l)$,  for any $k^{'}\in K$, using the $G$-invariance of $d_{X}$ we have $d_{X}(o, k^{'}x)=d_{X}(eK, k^{'}gK)=d_{X}(o, gK)<l$ which shows that $B(o,l)$ is left $K$-invariant, proving that $f_{\delta,j}$ also vanishes on $B(o,l)$.   Now writing $x=ka_r$ and making little abuse of notation we denote $$f_{\delta,j}(r)=\int_{K}f(k'a_r)Y_{\delta,j}(k'M)dk'.$$ It follows that $f_{\delta,j}(r)$ vanishes on a neighbourhood of $0$.  Using this notations, integrating the RHS of \ref{p2} in polar coordinates, we have 
	\begin{equation}
	\label{p3}
     F_{\delta,j}(\lambda)=\int_{0}^{\infty}f_{\delta,j}(r)\Phi_{\lambda,\delta}(a_r)\tilde{w}_{\alpha,\beta}(r)dr 
	\end{equation}  where recall that the weight  $\tilde{w}_{\alpha,\beta}$ is given by $\tilde{w}_{\alpha,\beta}(r)=(2\sinh r)^{2\alpha+1}(2\cosh r)^{2\beta+1}.$
	Now as mentioned above, $\Phi_{\lambda,\delta}$'s are known explicitly in terms of Jacobi functions :  
	$$\Phi_{\lambda,\delta}(a_r)=Q_{\delta}(i\lambda+\rho)(\alpha+1)_p^{-1}(\sinh r)^p(\cosh r)^q \varphi_{\lambda}^{\alpha+p,\beta+q}(r) $$ for some integers $p$ and $q$.
	Now recalling the definition of $\tilde{w}_{\alpha,\beta}$ and writing  $$\widetilde{f}_{\delta,j}(r)=\frac{4^{-(p+q)}}{(\alpha+1)_p}f_{\delta,j}(r)(\sinh r)^{-p}(\cosh r)^{-q} $$
	from \ref{p3} we have 
	\begin{equation}
	\widetilde{F}_{\delta,j}(\lambda)=\int_{0}^{\infty}\widetilde{f}_{\delta,j}(r)\varphi_{\lambda}^{\alpha+p,\beta+q}(r)\tilde{w
	}_{ \alpha+p,\beta+q}(r)dr 
	\end{equation}
	where $\widetilde{F}_{\delta,j}(\lambda)=Q_{\delta}(i\lambda+\rho)^{-1}F_{\delta,j}(\lambda).$ Hence it is clear that $\widetilde{F}_{\delta,j}(\lambda)$ represents the Jacobi transform of $(\alpha+p,\beta+q)$ of the function $\widetilde{f}_{\delta,j}. $ Hence in view of the inversion formula for the Jacobi transform we get 
	$$\widetilde{f}_{\delta,j}(r)=\frac{1}{2\pi}\int_{0}^{\infty}\widetilde{F}_{\delta,j}(\lambda)\varphi_{\lambda}^{\alpha+p,\beta+q}(r)|c_{\alpha+p,\beta+q}(\lambda)|^{-2}d\lambda.$$ Now considering the Jacobi operator $\mathcal{L}_{\alpha+p,\beta+q}$ with parameters $\alpha+p,\beta+q$, from the Plancherel formula we obtain 
	$$\|\mathcal{L}^m_{\alpha+p,\beta+q}\widetilde{f}_{\delta,j}\|^2_{L^2(\R^+, \tilde{w}_{ \alpha+p,\beta+q}(r)dr)}= C\int_{0}^{\infty}(\lambda^2+d^2)^{2m}|\widetilde{F}_{\delta,j}(\lambda)|^2|c_{\alpha+p,\beta+q}(\lambda)|^{-2}d\lambda$$ 
	where $d=\alpha+\beta+p+q+1.$ But from the hypothesis we have $$|\widetilde{F}_{\delta,j}(\lambda)|\leq C_{\delta,j} e^{-\lambda \theta(\lambda)},\ \lambda>0.$$
	Hence we have 
	\begin{equation}
\|\mathcal{L}^m_{\alpha+p,\beta+q}\widetilde{f}_{\delta,j}\|^2_{L^2(\R^+, \tilde{w}_{ \alpha+p,\beta+q}(r)dr)}\leq  C_{\delta,j}\int_{0}^{\infty}(\lambda^2+d^2)^{2m}e^{-2\lambda \theta(\lambda)} |c_{\alpha+p,\beta+q}(\lambda)|^{-2}d\lambda.
	\end{equation}
	Now under the assumption that $\theta(\lambda)\geq 2\lambda^{-1/2}$ for $\lambda\geq1$, it is a routine matter to check that   $\|\mathcal{L}^m_{\alpha+p,\beta+q}\widetilde{f}_{\delta,j}\|_2$ satisfies the Carleman condition, see e.g., \cite{BPR}. Since $\widetilde{f}_{\delta,j}$ vanishes in a neighbourhood of zero, from Theorem \ref{chernoffJ} we conclude that $\widetilde{f}_{\delta,j}=0.$  But this is true for every $\delta\in \widehat{K}_0$ and any $1\leq j\leq d_{\delta}$. Hence $f=0.$ Now we consider the general case. 
	
	Recall that $f$ vanishes on $B(o,l)$. Let $\Psi(t)=c(1+t)^{-1/2}  $, for $t>0$. Then it is easy to see that $\int_{1}^{\infty}\Psi(t)t^{-1}dt<\infty.$ Now by the rank one version of Theorem 4.2 of \cite{BPR}, there exists  a smooth $K $-biinvariant function on $G$   such that $ supp(g)\subset B(o,l/2) $ and its spherical transform satisfies 
	$$|\tilde{g}(\lambda)|\leq C e^{-|\lambda| \Psi(|\lambda|)} .$$
	Now let us consider the the function $F:=f\ast g.$ Then using the $G$-invariance of the Riemannian metric, it can be shown that $F$ vanishes on $B(o,l/2)$.   Also we see that, for   $kM\in K/M$ and any $\lambda $  
	$$\widetilde{F}(\lambda,kM)=\widetilde{f}(\lambda, kM)\tilde{g}(\lambda).$$ Hence it follows that for any $\lambda>0$
	$$\left|Q_{\delta}(i\lambda+\rho)^{-1}\int_{K}\widetilde{F}(\lambda,kM)Y_{\delta,j}(kM)dk\right|\leq C_{\delta,j}e^{- \lambda (\theta+\Psi)( \lambda )}$$
	and also $(\theta+\Psi)( \lambda )\geq 2\lambda^{-1/2}$ for $\lambda \geq 1.$ Therefore, from the first part of the proof, it follows that $F=0.$ So we have $\widetilde{f}(\lambda, kM)\tilde{g}(\lambda)=0 $ for all $\lambda $ and $kM$. But we know that $\tilde{g}(\lambda)$ is real analytic. Hence $f=0 $ proving the theorem.
	\end{proof}
	
\subsection{Ingham's theorem for spectral projections}	In this subsection we prove Theorem 1.1. We first consider the part of the theorem where the decay of the spectral projections is assumed to hold on the complement of the open set $ V $ over which $ f $ vanishes. This part is easily proved, thanks to  Theorem \ref{sspthm} proved in the previous subsection. For the other part we require some properties of the spherical means.

To begin with, we first describe the generalized spectral projections on $X$. Recall that the inversion formula for the Helgason Fourier transform says that  
$$f(x)=c_X\int_{-\infty}^{\infty}\int_{K/M}\widetilde{f}(\lambda,b)e ^{(i\lambda+\rho)A(x,b)}|c(\lambda)|^{-2}dbd\lambda.$$ We define the spectral projections as \begin{equation}\label{proj}
P_{\lambda}f(x):=\int_{K/M}\widetilde{f}(\lambda,b)e ^{(i\lambda+\rho)A(x,b)}db.
\end{equation} Now in view of the above inversion formula we have 
\begin{equation}
\label{specd}
f(x)=c_X\int_{ 0}^{\infty}P_{\lambda}f(x)|c(\lambda)|^{-2}d\lambda.
\end{equation} Now recall that the elementary spherical functions are defined as
$$\Phi_\lambda(x)=\int_K e^{(i\lambda+\rho)A(x,kM)} dk $$ which together with the following formula (see Bray \cite{B})  
$$\Phi_{\lambda}(h^{-1}g)=\int_{K/M}e^{(-i\lambda+\rho)A(hK,b)}e^{(-i\lambda+\rho)A(gK,b)} db$$ yields $P_{\lambda}f=f\ast \Phi_{\lambda}.$  So, 
   $P_{\lambda}f $ are eigenfunctions of the Laplace-Beltrami operator $\Delta_X$ with eigenvalues $-(\lambda^2+\rho^2)$. Therefore, these are the generalized spectral projections associated to $ \Delta_X $ and the   equation \ref{specd} can be thought of as the resolution of the identity  with respect to the operator $\Delta_X$. For more details about generalized spectral projections, we refer the reader to Bray \cite{B}.

   \textbf{Proof of Theorem \ref{thm1.1}}:
	 From the above discussion we note that 
	\begin{equation}
	f\ast \Phi_\lambda(x)=\int_{K/M}e^{(-i\lambda+\rho)A(x,b)}\tilde{f}(\lambda,b)db 
	\end{equation}
	where $x=gK\in X.$
	But now from the definition of  the Helgason Fourier transform, it follows that  
	\begin{equation}
	\label{sp1}
	\int_{G/K}f\ast \Phi_\lambda(x)\bar{f}(x)dx=\int_{K/M}|\tilde{f}(\lambda,b)|^2db.
	\end{equation} Under the assumption that $\displaystyle \sup_{x \in V^c} |f \ast \Phi_\lambda(x)| \leq C e^{-\lambda  \theta(\lambda)} $ we see that 
	\begin{align}
    \int_{G/K}f\ast \Phi_\lambda(x)\bar{f}(x)dx= \int_{V^c}f\ast \Phi_\lambda(x)\bar{f}(x)dx \leq C e^{-\lambda  \theta(\lambda)}
		\end{align}
which along with the the identity \ref{sp1} yields 
	\begin{equation}
	 \int_{K/M}|\tilde{f}(\lambda,b)|^2db \leq e^{-\lambda  \theta(\lambda)} 
	\end{equation}
	This will guarantee the condition \ref{ssp} of the Theorem \ref{sspthm}. Hence by that theorem we conclude that $f=0.$

	Now in order to prove the remaining part, namely that the spectral projections cannot have that particular pointwise decay on $V$, we consider the spherical means $A_hf$ defined by  
	\begin{equation}\label{spher-mean}  A_hf(g) = \int_K  f(gkh ) dk\,  .\end{equation}
	Observe that $ A_hf(g) $ is a right $ K$-invariant function of $ g \in G $ and hence we can consider it as a function on the symmetric space $ X.$ Also since $f$ is right $K$-invariant, it can be easily checked that the function $F_g$ defined by $$F_g(h)=A_hf(g)$$ is  a $K$-biinvariant function on $G$. So, we have
	$$\widetilde{F_g}(\lambda)=\int_{G}\left(\int_{k}f(gkh ) dk\, \right)\Phi_{\lambda}(h)dh.$$  An application of Fubini yields
	  $$\widetilde{F_g}(\lambda)=\int_{K} \int_{G}f(gkh )    \Phi_{\lambda}(h )dh\,dk $$ which, by a change of variable transforms to 
	  $$\widetilde{F_g}(\lambda)=\int_{K} \int_{G}f( h )    \Phi_{\lambda}(k^{-1}g^{-1}h )dh\,dk. $$ Recalling the fact that $\Phi_{\lambda}(h)=\Phi_{\lambda}(h^{-1})$, $K$-biinvariance of $\Phi_{\lambda}$ gives \begin{equation}
	  \label{sph_f }
	  \widetilde{F_g}(\lambda)=\int_{G}f(h)\Phi_{\lambda}(h^{-1}g)=f\ast\Phi_{\lambda}(g).
	  \end{equation} Now in view of inversion formula for spherical Fourier transform we have the following spectral form :
	  \begin{equation}F_g(h)=C\int_{0}^{\infty}\Phi_{\lambda}(h)f\ast\Phi_{\lambda}(g)|c(\lambda)|^{-2}d\lambda.\end{equation} 
	  But since $\Phi_{\lambda}$'s are eigenfunctions of the Laplace-Beltrami operator $\Delta_X$ with eigenvalue  $-(|\lambda|^2+ \rho^2)$, we have 
	  $$\widetilde{\Delta_X^m F_g}(\lambda)=-(|\lambda|^2+ \rho^2)^mf\ast\Phi_{\lambda}(g).$$ 
	  Hence from  Plancherel theorem it follows that 
	  $$\|\Delta_X^m F_g\|_2^2= C \int_{o}^{\infty}(\lambda ^2+ \rho^2)^{2m}|f\ast\Phi_{\lambda}(g)|^2|c(\lambda)|^{-2}d\lambda.$$ Now from $x=gK\in V$, by hypothesis we obtain 
	  $$\|\Delta_X^m F_x\|_2^2\leq C \int_{o}^{\infty}(\lambda^2+ \rho^2)^{2m}e^{-2\lambda \theta(\lambda)}|c(\lambda)|^{-2}d\lambda.$$ So we can check that $\|\Delta_X^m F_x\|_2$ satisfies the Carleman condition. But since $f$ vanishes on $V$ and $x\in V$,  it follows that $F_x$ vanishes on a neighbourhood of identity. Hence by Chernoff's theorem (\cite[Theorem 1.3]{BPR}) we get $F_x=0.$ But then form \ref{sph_f } it follows that 
	  $$f\ast\Phi_{\lambda}(x)=0,\ \text{for all} ~x\in V$$ But since $f\ast \Phi_{\lambda}$ are eigen functions of the Laplace-Beltrami operator which is elliptic,   $f\ast \Phi_{\lambda}$ is real analytic. So vanishing on $V$ forces  $f\ast \Phi_{\lambda}$ to be zero identically. Since this is true for every $\lambda$, it follows that $f=0.$\qed
	 % \end{proof}
	\begin{rem}
		The above theorem is sharp in the   sense that if $\int_{ 1}^{\infty}\theta(t)t^{-1}dt<\infty$, there exists $f\in C^{\infty}_c(G//K)$ such that $|P_{\lambda}f(x)|\leq C e^{-\lambda\theta(\lambda)}.$ Indeed, by Theorem 4.2 in \cite{BPR}, there exist $f\in C_c^{\infty}(G//K)$ such that $|\tilde{f}(\lambda)| \leq C e^{-\lambda\theta(\lambda)}$ for all $\lambda>0.$ Now from the definition \ref{proj} it follows that $P_{\lambda}f(x)=\tilde{f}(\lambda)\Phi_{\lambda}(x).$   Using the fact that $|\Phi_{\lambda}(x)|\leq \Phi_0(x)\leq C$ (See Bray \cite{B}), we obtain $|P_{\lambda}f(x)|\leq C e^{-\lambda\theta(\lambda)} $ proving the claimed sharpness.
	\end{rem}
	
	We conclude this section by briefly describing a version of Ingham type theorem for right $K$-invariant functions on rank one semisimple Lie groups, which can be obtained as an immediate consequence of Theorem \ref{sspthm}. We need some preparations for that.
	
    For $\lambda\in \mathfrak{a}^{*}_{\C}$ which is just $\C$  in our case, we consider the irreducible representations $\pi_{\lambda}$ on the rank one semisimple Lie group $G$ under consideration acting on the Hilbert space $L^2(K/M)$, defined by
$$\pi_{\lambda}(g)f(k)=e^{(i\lambda+\rho)A(g,k)}f(k(g^{-1}k)) $$ where $g=k(g) expA(g)n(g)$ is  the Iwasawa decomposition of $ g \in G$. It is well known that $\pi_{\lambda}$ is unitary if and only if $\lambda$ is real (i.e., in $\mathfrak{a}^*$).  These are called the \textit{class-1 principle series representations}. We know that the group Fourier transform for right $K$-invariant function on $G$ takes the form
$$\pi_{\lambda}(f)=\int_Gf(g)\pi_{\lambda}(g)dg.$$ Now considering $f$ as a function on $X=G/K$, the Helgason Fourier transform of $f$ is related to the group Fourier transform via the relation 
$$\widetilde{f}(\lambda,b)= \pi_{\lambda}(f)Y_0(b),~ b\in K/M.$$ Here $Y_0$ is the function corresponding to the unit representation of $\widehat{K}_0.$ It is easy to see that for any $\delta\in \widehat{K}_0$ and $1\le j\le d_{\delta}$ we have 
$$\int_{K/M}\widetilde{f}(\lambda,b)Y_{\delta,j}(b)db=\langle \pi_{\lambda}(f)Y_{0}, Y_{\delta,j}\rangle.$$ This observation together with Theorem \ref{sspthm} yields the following result:
 \begin{cor} 
 	 Let $ \theta $ be a positive decreasing function defined on $ [0,\infty) $ that vanishes  at infinity.
 	Suppose $f$ is a right $K$- invariant, integrable function on $G$ which vanishes on an open neighbourhood of identity. Assume that for any $\delta\in \widehat{K}_0$ and $1\le j\le d_{\delta}$ 
 	$$\left|Q_{\delta}(i\lambda+\rho)^{-1}\langle \pi_{\lambda}(f)Y_{0}, Y_{\delta,j}\rangle\right|\leq C_{\delta,j}e^{-\lambda \theta(\lambda)}.$$ Then $\int_{ 1}^{\infty}\theta(t)t^{-1}dt=\infty$ implies $f=0.$ 
\end{cor} 
	
	\section{Spectral projections associated to Dunkl-laplacian}
In this section we prove an uncertainty principle for the spectral projections for the Dunkl-Laplacian. To begin with, we first describe the basic theory of Dunkl transform in the following subsection. 

\subsection{Background for Dunkl transform}
For $v(\neq 0)\in \mathbb{R}^n$, the reflection $r_{v}$ with respect to hyperplane perpendicular to $v$ is given  by $r_v(x):=x-2\left(\langle x,v\rangle/\|v\|^2\right)v,~x\in\mathbb{R}^n.$ Let $R$ be a reduced root system in $\mathbb{R}^n$ i.e., $R$ consists of finite number of non-zero vectors in $\R^n$ with the property that $r_u(v)\in R$ for any $u,v\in R$ and moreover if $u=av$ then $a=\pm1.$ Let us fix a set of positive roots $R^{+}.$ Suppose $G$ is a subgroup of the orthogonal group $O(n)$ generated by the reflections $\{r_v:v\in R\}.$ 

Let $\K:R^+\rightarrow [0,\infty)$ be a multiplicity function which is $G$ invariant. Associated to this root system $R$ and the multiplicity function $\K$, Dunkl considered the first order differential-difference operators defined by 
$$D_jf(x)=\frac{\partial f}{\partial x_j}+\displaystyle \sum_{v\in R^+}\K_v\frac{f(x)-f(r_v(x))}{\langle x,v\rangle}\langle v,e_j\rangle,~ 1\leq j\leq n$$ where $e_j$'s are standard unit vectors in $\R^n$. These operators commute with each other i.e., $D_jD_i=D_iD_j,~\forall i,j$. Let $V_{\K}$ denote the unique operator which intertwines the algebra generated by $D_j$'s and the algebra of partial differential operators, determined by $$V_{\K}\mathcal{P}_m\subset \mathcal{P}_m,~ V_{\K}1=1~\text{and}~ D_jV_{\K}=V_{\K}\frac{\partial}{\partial x_j},~1\leq j\leq n  $$ where $\mathcal{P}_m$ denotes the space of homogeneous polynomials of degree $m$. Given $x,y\in\R^n$, consider the function $E(x,y):=V_{\K}(e^{\langle .,y\rangle})(x)$. It is known that for fixed $y$, the function $E(.,y)$ is the unique solution of $D_jf(x)=\langle y,e_j\rangle f(x),~f(0)=1$. Moreover, $E$ can be extended to $\C^n\times \C^n$ holomorphically. Several important properties of this function are listed in the following proposition:
\begin{prop}
	For any $z,w\in\C^n$ and $\lambda\in\C$, $E(z,w)=E(w,z)$ , $E(\lambda z,w)=E(z,\lambda w)$ and  $E$ satisfies the estimate $|E(z,w)|\leq e^{|z|.|w|} $ for all $z,w\in C^n.$ Moreover, 
	$$c_{\kappa}\int_{ \R^n}E(z,x)E(w,x)h^2_{\K}(x)e^{-|x|^2/2}dx=e^{z^2+w^2}E(z,w)$$ where 
	the constant $c_{\K}$ is defined by $c_{\K}^{-1}:=\int_{ \R^n}e^{-|x|^2/2}h^2_{\K}(x)dx.$
\end{prop}
In the above proposition  weight function $h^2_{\K}$ is defined by  
$$h^2_{\K}(x):=\prod_{v\in R^+}|\langle x,v\rangle|^{2\K_v},~ x\in \R^n.$$ It is easy to see that this function is positive homogeneous of degree $2\gamma $ where $\gamma:=\sum_{v\in R^+}\K_v$.  Moreover, it is invariant under the reflection group $G$.  We are now ready to define Dunkl transform. Note that the preceding proposition suggests that $E(x,iy)$ plays the role of $e^{i\langle x,y\rangle}$ in Euclidean harmonic analysis. 

Given $f\in L^1(\R^n, h^2_{\K}(x)dx)$, the Dunkl transform of $f$ is defined as 
$$\mathcal{F}_kf(\xi)=c_{\K}\int_{ \R^n}f(x)E(x,-i\xi)h^2_{\K}(x)dx,~ \xi\in\R^n.$$ it is worth pointing out that when $\K=0$, then $V_{\K}=id$ and $h^2_{\K}=1$, thus the Dunkl transform coincides with the Fourier transform. In this sense, this serves as a generalisation of Euclidean Fourier transform.  We also have Plancherel and inversion formula for Dunkl transform. We record those in the following theorem:
\begin{thm} \begin{enumerate}
		\item (Plancherel) $\mathcal{F}_{\K}$ extends to the whole of $L^2(\R^n,h^2_{\K}(x)dx)$ as an isometry onto itself. 
		\item (Inversion) For $f$ and $\mathcal{F}_{
		K}f\in L^1(\R^n,h^2_{\K}(x)dx)$ we have the following inversion formula 
	$$f(x)=c_{\K}\int_{ \R^n}\mathcal{F}_{\K}f(\xi)E(ix,\xi)h^2_{\K}(\xi)d\xi,~ \text{a.e.}~ x\in\R^n.$$
	\end{enumerate}
\end{thm} 
 The Dunkl-Laplacian  is defined as $\Delta_{\K}:=\sum_{j=1}^nD_j^2.$ Let $H_m(h^2_{\K})$ denote the space of all $h$-harmonic polynomials of degree $m$ i.e., all those $P\in\mathcal{P}_m$ such that $\Delta_{\K}P=0.$ The spherical $h$-harmonics are restriction of $h$-harmonic polynomials to $S^{n-1}$. We consider the following inner product 
 $$(f,g)_{\K}=a_{\K}\int_{S^{n-1}}f(\om)g(\om)h^2_{\K}(\om)d\sigma(\om)$$ where $a_{\K}^{-1}=\int_{S^{n-1}}h^2_{\K}(\om)d\sigma(\om).$ With respect to this inner product, the space $L^2(S^{n-1},h^2_{\K}d\sigma)$ can be decomposed as   $L^2(S^{n-1},h^2_{\K}d\sigma)=\bigoplus_{m=0}^{\infty} H_m(h^2_{\K}).$ 
   We also have the following very useful formula proved in \cite{CD}
   \begin{prop}
   	\label{df}
   	Let $r>0$. Given $S^h_m\in H_m(h^2_{\K})$, for $x=|x|x'\in\R^n$ we have
   	$$ a_{\K}\int_{S^{n-1}}S^h_m(\om)E(x,-ir\om)h^2_{\K}(\om)d\sigma(\om)=(-i)^m2^{\lambda_{\K}}S^h_{m}(x')(r|x|)^{-\lambda_{\K}}J_{m+\lambda_{\K}}(r|x|),$$ where $\lambda_{\K}=\gamma+(n-2)/2$. Moreover, the function given by the above integral   is an eigenfunction of $\Delta_{\K}$ with eigenvalue $-r^2$. 
   \end{prop}
 As an immediate consequence of this result, we conclude that the Dunkl transform of reasonable radial function $f$ on $\R^n$ is given by the Hankel transform $\mathcal{H}_{\lambda_{\K}}f .$
 
 For suitable functions $f$, the generalised translation operator is defined by 
 $$\tau_yf(x):=c_{\K}\int_{ \R^n}\mathcal{F}_{\K}f(\xi)E(ix,\xi)E(iy,\xi)h^2_{\K}(\xi)d\xi,~x\in\R^n.$$  Using this one can define the Dunkl convolution as 
 $$f\ast_{\K}g(x)=c_{\kappa}\int_{\R^n}f(y)\tau_x \check{g}(y)h^2_{\K}(y)dy. $$ It can be easily checked that $\mathcal{F}_{\K}(f\ast_{\K}g)=\mathcal{F}_{\K}f \mathcal{F}_{\K}g.$ For more about this translation and convolution, we refer the reader to work of Thangavelu-Yu \cite{TY}.
 
 \subsection{ Ingham's theorem for the Dunkl transform} In this subsection we prove a version  of Ingham type theorems for Dunkl transform. We start with constructing a compactly supported function whose Dunkl transform has Ingham type decay. In order to do so, we need a Paley-Wiener type theorem for Hankel transform. For that we first define a function space  $\mathcal{H}$ as follows:   we say that  a function $g$ on $\C$ belongs to $\mathcal{H}$ if $g$  is an even entire function and there are positive constants $A$ and $C_m$ such that for all $z\in\C$ and for all $m=0,1,2,...$, $g$ satisfies the following: 
   $$|g(z)|\leq C_m(1+|z|)^{-m}e^{A|\Im(z )|}.$$
   
   We have the following Paley-Wiener type theorem
    described in Koornwinder \cite{K}.
   \begin{thm} 
   	\label{PWH}
   Let $ \alpha >-1/2.$ The mapping $f\rightarrow \mathcal{H}_{\alpha}f$ is a bijection from $C_0^{\infty}(\R)_e$, the set of all compactly supported even smooth function on $\R$, onto $\mathcal{H}.$  
   \end{thm}
As a consequence we obtain 
\begin{prop}
	\label{h_e}
	 Let $ \theta $ be a positive decreasing function defined on $ (0,\infty) $ that vanishes  at infinity. Assume that $ \int_1^\infty \theta(t) \, t^{-1} dt <\infty.$ Then there exist a nontrivial radial smooth function $f$ on $\R^n$, supported in a small neighbourhood of the origin, satisfying $|\mathcal{F}_{\K}f(\xi)|\leq C e^{-|\xi|\theta(|\xi|)}.$  
\end{prop}
\begin{proof}
	Since $ \int_1^\infty \theta(t) \, t^{-1} dt < \infty$, by Ingham's theorem (see \cite{I}) there exist an even  compactly supported smooth function $f_0$ on $\R$,    whose Fourier transform satisfies
	\begin{equation}
	\label{cs-1}
	|\hat{f_0}(\xi)|\leq Ce^{-|\xi|\theta(|\xi|)},~ \xi\in \R. 
	\end{equation} 
	But in view of the Paley-Wiener theorem for the Fourier transform (See e.g., \cite{stein}), $\hat{f_0}\in \mathcal{H}.$ Therefore, applying Theorem \ref{PWH} with $\alpha=\lambda_{\K}$  we get $g\in C_0^{\infty}(\R)_e$ such that $\mathcal{H}_{\lambda_{\K}}g=\hat{f_0}.$ Now define $f$ on $\R^n$ by $f(x)=g(|x|),~ x\in\mathbb{R}^n.$ Finally using the estimate \ref{cs-1}, recalling the fact that Dunkl transform of radial function is given by Hankel transform of type $\lambda_{\K}$ we are done.  
	\end{proof}

 \begin{thm}
 	\label{d_I1}
 	Let $f\in L^1(\R^n, h^2_{\K}(x)dx)$ be such that $f$ vanishes on an open neighbourhood of the origin $V\subset \R^n.$ For each non-negative integer $m$ and $S^h_{m}\in H_m(h^2_{\K})$ assume that 
 	$$\left|\lambda^{-m}\int_{ S^{n-1}}\mathcal{F}_kf(\lambda,\omega)S^h_m(\om)h^2_{\K}(\om)d\sigma(\om)\right|\leq C_me^{- \lambda \theta( \lambda )},~\lambda>0$$ where $\theta$ is a positive  decreasing function on $[0,\infty)$ which vanishes at infinity. Then, if $\int_{1}^{\infty}\theta(t)t^{-1}dt=\infty$,  $f$ is identically zero. 
 \end{thm}
   \begin{proof} Without loss of generality we can assume that $f$ vanishes on a ball $B(0,l)$ of radius $l$ for some $l>0$.  
   	Note that using the definition of Dunkl transform we have 
   	$$\int_{ S^{n-1}}\mathcal{F}_kf(\lambda,\omega)S^h_m(\om)h^2_{\K}(\om)d\sigma(\om)=c_{\K}\int_{ \R^n}\left(\int_{ S^{n-1}}E(-ix,\lambda\om)S^h_m(\om)h^2_{\K}(\om)d\sigma(\om)\right)f(x)h^2_{\K}(x)dx. $$ Now in view of Proposition \ref{df}, we have 
   	\begin{equation}
   	\int_{ S^{n-1}}\mathcal{F}_kf(\lambda,\omega)S^h_m(\om)h^2_{\K}(\om)d\sigma(\om)=\int_{ \R^n} (\lambda|x|)^{-\lambda_{\K}}J_{m+\lambda_{\K}}(\lambda|x|)S_m^h\left(\frac{x}{|x|}\right)h^2_{\K}(x)dx.
   	\end{equation}
   	Now writing $f_m(r)=\int_{ S^{n-1}}f(rw)S^h_m(w)h^2_{\K}(w)d\sigma(w)$ where $x=rw,~ w\in S^{n-1}$,the polar coordinate representation transforms the above equation to 
   	\begin{equation}
   	\int_{ S^{n-1}}\mathcal{F}_kf(\lambda,\omega)S^h_m(\om)h^2_{\K}(\om)d\sigma(\om)=C\lambda^{m}\int_{ 0}^{\infty}f_{m}(r)r^{-m}\frac{J_{\lambda_{\K}+m}(\lambda r)}{(\lambda r)^{\lambda_{\K}+m} }r^{2(\lambda_{\K}+m)+1}dr. 
   	\end{equation}
   	Again letting $g(r)=f_m(r)r^{-m}$, we have 
   	\begin{equation}
   	\lambda^{-m}	\int_{ S^{n-1}}\mathcal{F}_kf(\lambda,\omega)S^h_m(\om)h^2_{\K}(\om)d\sigma(\om)= C \mathcal{H}_{\lambda_{\K}+m}g(\lambda).
   	\end{equation}
   	Now by the hypothesis, it is clear that $g$ vanishes on a neighbourhood of zero. Let us denote $\lambda_{\K}+m$ by $\alpha$. Hence using the Plancherel theorem for Hankel transform (See Theorem \ref{hpi}), for any $k\in \mathbb{N}$ we have 
   	$$\|\Delta^k_{\alpha,a} g\|_2^2= C \int_{ 0}^{\infty}(\lambda^2+a^2)^{2k}|\mathcal{H}_{\alpha}g(\lambda)|^2\lambda^{2\alpha+1}d\lambda \leq C \int_{ 0}^{\infty}(\lambda^2+a^2)^{2k} e^{-2\lambda\theta(\lambda)}\lambda^{2\alpha+1}d\lambda. $$ Under the assumption that $\theta(\lambda)\geq 2\lambda^{-1/2}$ for $\lambda \geq1$, it is a routine matter to check that $\|\Delta^k_{\alpha,a} g\|_2$ satisfies the Carleman condition (see \cite{BPR}) and hence by Theorem \ref{chernoffH}, we obtain $g=0.$ So $f_m=0$ i.e., $(f(r.), S^h_m)_{\K}=0$ which is true for any $h$-spherical harmonics. Therefore, $f=0$. 
   	
   	To treat the general case we take $\psi(t)=c(1+t)^{-1/2}$. Then it can be easily checked that $\int_{1}^{\infty}\psi(t)t^{-1}dt<\infty.$ By Proposition \ref{h_e}, there exist a radial $h\in C^{\infty}_c(\R^n)$ supported in $B(0,l/2)$ such that $|\mathcal{F}_{\K}h(\xi)|\leq C e^{-|\xi|\psi(|\xi|)}.$ Define $F:=f\ast_{\K}h$.  Then in view of \cite[Proposition 3.13]{TY}, $F$ vanishes on $B(0,l/2)$. Moreover, since $\mathcal{F}_{\K}(F)(\lambda,\omega)=\mathcal{F}_kf(\lambda,\om)\mathcal{F}_{\K}h(\lambda)$, we see that 
   	$$\left|\lambda^{-m}\int_{ S^{n-1}}\mathcal{F}_kF(\lambda,\omega)S^h_m(\om)h^2_{\K}(\om)d\sigma(\om)\right|\leq C_me^{-\lambda(\theta+\psi)(\lambda )}$$ where $(\theta+\psi)( \lambda )\geq 2\lambda^{-1/2}$ for $\lambda \geq 1$. Hence by the previous case, $F=0$ and since $h$ is a nontrivial function, we conclude that $f=0$. This completes the proof. 
   	\end{proof}

\subsection{Ingham's theorem for the spectral projections} 
For $ \lambda >0, \,  \omega \in S^{n-1} $ we write the Dunkl transform in the form
$$  \mathcal{F}_{\K}f(\lambda\omega) = c_{\kappa}  \int_{\R^n} f(x) E( x,-i \lambda \om) h^2_{\K}(x)dx .$$ Then we can rewrite the inversion formula as
$$f(x)=c_{\kappa}\int_{ 0}^{\infty}\left(\int_{S^{n-1}}\mathcal{F}_{\K}f(\lambda\omega)E(ix,\lambda\om)h^2_{\K}(\om)d\sigma(\om)\right)\lambda^{2\lambda_{\K}+1}d\lambda.$$ But in view of the Proposition \ref{df}
$$\varphi_{\K,\lambda}(x):=\int_{S^{n-1}} E(ix,\lambda\om)h^2_{\K}(\om)d\sigma(\om)= 2^{\lambda_{\K}}a_{\K}^{-1} \frac{J_{\lambda_{\K}}(\lambda |x|)}{(\lambda |x|)^{\lambda_{\K}}}$$ which allows us to write the inversion formula in the following form:
$$f(x)=c_{\K}\int_{ 0}^{\infty}f\ast_{\K}\varphi_{\K,\lambda}(x)\lambda^{2\lambda_{\K}+1}d\lambda.$$
With the above definition, it is easy to see that 
$$  c_{\kappa}  \int_{\R^n} f\ast_{\K} \varphi_{\K,\lambda}(x) \bar{f}(x) h^2_{\K}(x)dx =  \int_{S^{n-1}} |\mathcal{F}_{\K}f(\lambda\omega)|^2 \, h^2_{\K}(\om)d\sigma(\omega) $$
and consequently the Plancherel theorem takes the form
$$  c_{\K}  \int_0^\infty  \langle  f\ast_{\K} \varphi_{\K,\lambda}, f \rangle  \lambda^{2\lambda_{\K}+1}  d\lambda =  \int_{\R^n} |f(x)|^2 h^2_{\K}(x)dx.$$
We remark that given a reasonable function $f$,   the function
\begin{equation}
\label{proj_d}
f\ast_{\K}\varphi_{\K,\lambda}(x)= \int_{S^{n-1}}\mathcal{F}_{\K}f(\lambda\omega)E(ix,\lambda\om)h^2_{\K}(\om)d\sigma(\om)
\end{equation} 
is an eigenfunction of $\Delta_{\K}$ with eigenvalue $-\lambda^2$ (See Proposition \ref{df}). Therefore, these are just the generalized spectral projections associated to the Dunkl-Laplacian. We have the following Ingham type uncertainty principle for these generalized spectral projections.\\

   \textbf{Proof of Theorem \ref{thm1.3}:}
	Assume that   $ \sup_{x \in V^c} |f \ast \varphi_{\K,\lambda}(x)| \leq C e^{-\lambda  \theta(\lambda)} $. Then as $f$ vanishes on $V$ 
	\begin{equation}
	 \int_{S^{n-1}} |\mathcal{F}_{\K}f(\lambda,\omega)|^2 \, h^2_{\K}(\om)d\sigma(\omega)=c_{\kappa}  \int_{V^c} f\ast_{\K} \varphi_{\K,\lambda}(x) \bar{f}(x) h^2_{\K}(x)dx  \leq C e^{-\lambda\theta(\lambda)}. 
	\end{equation}
	Now expanding $\mathcal{F}_{\K}f(\lambda,.)$ in terms of the $h-$Spherical harmonics, we see that the hypothesis of Theorem \ref{d_I1} is satisfied. Hence $f=0.$

	To treat the other case we consider the Dunkl spherical means $ f \ast_{\K} \mu_r $ which  has the integral representation
	$$   f\ast_{\K} \mu_r(x)  = c_{\K} \int_0^\infty  f \ast_{\K} \varphi_{\K,\lambda}(x) \, \varphi_{\K,\lambda}(r) \,\lambda^{2\lambda_{\K}+1} d\lambda. $$ 
	For any fixed $ x \in V $ we consider the radial  function $ F_x(r) =  f \ast_{\K} \mu_{r}(x) $ which vanishes in a neighbourhood of zero. Since $ \varphi_\lambda(r) $ is an eigenfunction of $ \Delta_{\lambda_{\K},a} $ with eigenvalue $ -(\lambda^2+a^2) $ it follows that
	$$  \| \Delta_{\lambda_{\K},a}^m F_x\|_2^2 = \int_0^{\infty}  (\lambda^2+a^2)^{2m} |f \ast_{\K} \varphi_{\K,\lambda}(x)|^2  \lambda^{2\lambda_{\K}+1} d\lambda \leq C \int_0^{\infty} (\lambda^2+a^2)^{2m}  e^{-2\lambda \, \theta(\lambda)} \lambda^{2\lambda_{\K}+1} d\lambda. $$
	We can check that the sequence $ \| \Delta_{\lambda_{\K},a}^m F_x \|_2 $ satisfies the Carleman condition and hence by  Chernoff's theorem we conclude $  f \ast_{\K} \mu_r(x) = 0 $ for all $ r >0.$ 
	But then the Plancherel theorem for the Hankel transform gives
	$$ \int_0^\infty | f \ast \varphi_{\K,\lambda}(x)|^2  \, \lambda^{2\lambda_{\K}+1} d\lambda =0 $$ 
	for all $ x \in V. $ Thus, the real analytic function $ f \ast \varphi_{\K,\lambda} $ vanishes on $ V $ and hence vanishes everywhere. As this is true for any $ \lambda $ we get $ f =0.$
	\qed
\begin{rem}
	This theorem is sharp in the sense that when $\int_{ 1}^{\infty}\theta(t)t^{-1}dt<\infty$, there exist $f\in C^{\infty}_c(\R^n)$ which is radial and satisfies $|f\ast_{\K}\varphi_{\K, \lambda}(x)|\leq C e^{-\lambda\theta(\lambda)}$ for all $\lambda>0.$ The proof of this is not very difficult. In fact, by Proposition \ref{h_e}, there exist compactly supported smooth radial function $f$ on $\R^n$ satisfying  $|\mathcal{F}_{\K}f(\xi)|\leq C e^{-|\xi|\theta(|\xi|)} $ for all $\xi\in \R^n.$ Now by \ref{proj_d}, it follows that $f\ast_{\K}\varphi_{\K, \lambda}(x)=\mathcal{F}_{\K}f(\lambda)\varphi_{\K, \lambda}(x)$. Finally we use the boundedness of Bessel functions  to conclude that $|f\ast_{\K}\varphi_{\K, \lambda}(x)|\leq C e^{-\lambda\theta(\lambda)}$.
\end{rem}
  \begin{rem}
  		 Recall that when $\K=0$  the above analysis reduces to the usual Fourier analysis on $\R^n.$ Consequently, the generalized spectral projections associated to the Laplacian $\Delta:=\sum_{j=1}^{n}\frac{\partial^2}{\partial x_j^2}$ on $\R^n $ take the form 
  		 $$f\ast \varphi_{\lambda}(x)= (2\pi)^{-n/2}\int_{S^{n-1}}\hat{f}(\lambda\om)e^{i\lambda x.\om}d\sigma(\om) $$ where $\hat{f}(\lambda\om)$ denotes the Euclidean Fourier transform written in polar coordinates. 
  		 So, as an immediate consequence of the above theorem we have the following uncertainty principle for the generalized spectral projections:  
  		 \begin{cor}
  		 	 	 Let $ \theta $ be a positive decreasing function defined on $ [0,\infty) $ that vanishes  at infinity. Assume that $ \int_1^\infty \theta(t) \, t^{-1} dt = \infty.$ Let $ f \in L^1(\R^n, dx) $ be a nontrivial function vanishing on an open set $ V.$  Then 
  		 	the estimate  $ \sup_{x \in V} |f \ast \varphi_{ \lambda}(x)| \leq C e^{-\lambda \,  \theta(\lambda)} $ uniformly for all $\lambda$ cannot hold. If $ V $ contains an open neighbourhood of the origin, then the uniform estimate $ \sup_{x \in V^c} |f \ast \varphi_{ \lambda}(x)| \leq C e^{-\lambda \,  \theta(\lambda)} $ can also not hold.
  		 \end{cor} 
  \end{rem}

\section*{Acknowledgments}The authors wish to thank the referee for the careful reading of the manuscript and for many useful suggestions. The first author is supported by Int. Ph.D. scholarship from Indian Institute of Science.
The second author is supported by  J. C. Bose Fellowship from D.S.T., Govt. of India.
%%%%%%%%%%%%%%%%%%%%%%%%%%%%%%%%%%%%%%%%%%%%%%%%%%%%%%

\end{document}